%% file: oltd.tex
\documentclass[letterpaper, 10pt, conference]{ieeeconf}
\IEEEoverridecommandlockouts
\usepackage{amssymb}
\usepackage{amsmath}
\usepackage{algorithm}
\usepackage{algorithmic}
\usepackage{xcolor}
\usepackage{verbatim}

\usepackage{amsthm}
\usepackage{xcolor}
\usepackage{cite}
\usepackage{hyperref}
\usepackage{booktabs}
\usepackage{tikz}
\usepackage{pgfplots}
\usepackage{pgfplotstable}
\usepgfplotslibrary{fillbetween}

\newcommand{\E}{\mathbb E}
\newcommand{\F}{\mathcal F}
\newcommand{\K}{\mathcal K}
\renewcommand{\L}{\mathcal L}
\newcommand{\N}{\mathcal N}
\renewcommand{\S}{\mathbb S}
\newcommand{\R}{\mathbb R}
\newcommand{\X}{\mathcal X}
\newcommand{\bmat}[1]{\begin{bmatrix} #1 \end{bmatrix}}

\DeclareMathOperator{\Tr}{Tr}
\DeclareMathOperator{\Diag}{Diag}
\DeclareMathOperator*{\argmin}{argmin}

\newcommand{\onote}[1]{{\color{black} #1}}

\newtheorem{theorem}{Theorem}[section]
\newtheorem{proposition}[theorem]{Proposition}
\newtheorem{lemma}[theorem]{Lemma}

\theoremstyle{definition}
\newtheorem{definition}[theorem]{Definition}
\newtheorem{assumption}[theorem]{Assumption}
\theoremstyle{remark}

\usetikzlibrary{external}
\makeatletter
\newcommand\fs@betterruled{%
  \def\@fs@cfont{\bfseries}\let\@fs@capt\floatc@ruled
  \def\@fs@pre{\vspace*{5pt}\hrule height.8pt depth0pt \kern2pt}%
  \def\@fs@post{\kern2pt\hrule\relax}%
  \def\@fs@mid{\kern2pt\hrule\kern2pt}%
  \let\@fs@iftopcapt\iftrue}
\floatstyle{betterruled}
\restylefloat{algorithm}
\makeatother

\title{\LARGE \bf
Learning Optimal Team-Decisions
}

\author{Olle Kjellqvist and Ather Gattami
\thanks{O. Kjellqvist is with the Department of Automatic Control, Lund University. A. Gattami is with AI Sweden, Stockholm, Sweden. This project has received funding from the European Research Council (ERC) under the European Union’s Horizon 2020 research and innovation programme under grant agreement No 834142
(ScalableControl).
\copyright 2022 IEEE. Personal use of this material is permitted. Permission from IEEE must be
obtained for all other uses, in any current or future media, including
reprinting/republishing this material for advertising or promotional purposes, creating new
collective works, for resale or redistribution to servers or lists, or reuse of any copyrighted
component of this work in other works.}
}
\begin{document}

\maketitle
\thispagestyle{empty} 
\pagestyle{empty}

\begin{abstract}
    In this paper, we treat linear quadratic team decision problems, where a team of agents minimizes a convex quadratic cost function over $T$ time steps subject to possibly distinct linear measurements of the state of nature. We assume that the state of nature is a Gaussian random variable and that the agents do not know the cost function nor the linear functions mapping the state of nature to their measurements. We present a gradient-descent based algorithm with an expected regret of $O(\log(T))$ for full information gradient feedback and $O(\sqrt(T))$ for bandit feedback. In the case of bandit feedback, the expected regret has an additional multiplicative term $O(d)$ where $d$ reflects the number of learned parameters.
\end{abstract}

\section{Introduction}
\label{sec:intro}
\input{sections/intro}

\section{Preliminaries}
\label{sec:prelims}
\input{sections/prelims}

\section{Problem Formulation}
\label{sec:problem}
\input{sections/problem}
\section{Properties of Stochastic Team Decisions}
\label{sec:stoch}
\input{sections/stoch}

\section{Learning Optimal Team Decisions}
\label{sec:learning}
\input{sections/learning}


\section{Numerical example}
\label{sec:example}
\input{sections/example.tex}
\section{Conclusions and Future Research}
\label{sec:conclusions}
\input{sections/conclusions}

\bibliography{oltd}
\bibliographystyle{IEEEtran}
\section{Appendix}
\input{sections/appendix}
\end{document}

%% file: sections/intro.tex
Team decision problems originate from economics, where optimal decentralized decisions in organizations were studied in the papers by Marschak \cite{marschak}, and Radner \cite{radner} under stochastic settings. In these studies, the agents in the team \textit{know the problem parameters}. The agents use the information of the problem parameters to find the optimal decentralized decision. Decentralized decisions only depend on local measurements of the state of nature, where the measurements of the agents are typically not identical. Gattami \cite{Gattami2012Robust} studied linear quadratic robust team decision problems and showed that optimal decisions are linear and can be found by solving a convex (in fact, semi-definite) optimization problem.  
\onote{Team-decision theory has been helpful in understanding distributed control research~\cite{Mahajan2012Information}.
Witsenhausens famous counterexample~\cite{Witsenhausen1968Conterexample} established that linear decisions are not always optimal for distributed LQG problems and sparked an interest into research of team problems in the control community.
Ho and Chu~\cite{Ho1972Team} showed how linear-quadratic problems with partially nested information can be rewritten as static team-decision problems of the type in this paper and Witsenhausen showed that a general class of dynamic team decision problems can be reduced to static ones via a change of measures~\cite{Witsenhausen1988EquivalentSC}.
Static reductions for more exotic information structures is still an active research field~\cite{Gupta2014Existence, Sanjari2021Policy}.}

In this article, we study \emph{learning} of optimal decentralized decisions with linear information constraints and quadratic cost functions in the stochastic setting, \textit{without the knowledge of the problem parameters}. We consider learning with gradient feedback and bandit feedback. We study expected regret against the optimal policy in hindsight. Our key contributions are:
\begin{itemize}
    \item We propose a first and a zeroth-order algorithm to learn optimal decentralized decisions with linear information constraints and quadratic cost functions through repeated interactions.
    \item We extend the regret analysis of online gradient descent to the case with a possibly unbounded gradient oracle that has bounded second-moment.
    \item We show that our algorithms have expected regret bounded by $O(log(T))$ if the gradient is observed and $O(\sqrt{T})$ if only the incurred loss is observed in each step.
\end{itemize}

\subsection{Outline}
We give some background and establish notation in Section~\ref{sec:prelims} and formalize the learning problem in Section~\ref{sec:problem}. Section~\ref{sec:stoch} is devoted to properties of decentralized stochastic team-decision problems. Section~\ref{sec:learning} contains our extension to the regret analysis of online gradient descent and its application to the stochastic team-decision problem. In Section~\ref{sec:conclusions} we summarize our conclusions and give directions for future research. 

%% file: sections/prelims.tex
\subsection{Notation}
We denote the space of $n$-dimensional real-valued vectors by $\R^n$ and real-valued matrices with $m$ rows and $n$ columns by $\R^{m\times n}$. 
For a vector $x\in \R^n$, $\|x\|_2 = \sqrt{x^\top x}$ denotes the Euclidean norm and $A^\top$ denotes the transpose of a matrix $A$. $\Tr M$ denotes the trace of a square matrix $M$. For matrices $A, B \in \R^{m\times n}$, we denote the operator norm of $A$ as $\|A\|_2 = \max_{\|x\|_2 = 1}\|Ax\|_2$, the Frobenius inner product as $\langle A, B \rangle_F = \Tr A^\top B$, and the Frobenius norm as $\|A\|_F = \sqrt{\langle A, A \rangle_F}$. We denote the smallest singular value of a matrix $A \in \R^{m \times n}$ by $\sigma_\text{min}(A)$. The set of real-valued $n\times n$-dimensional symmetric matrices is denoted as $\S^n$. $\S^n_+$ and $\S^n_{++}$ refer to the sets of $n\times n$-dimensional of positive semi-definite and positive definite matrices, respectively. For a matrix $A \in \R^{(m_1 + \cdots + m_M)\times (n_1 + \cdots + n_N)}$, $[A]_i\in \R^{m_i\times (n_1 + \cdots + n_N)}$ denotes the $i$th block row and $[A]_{ij} \in \R^{m_i \times n_j}$ denotes the block element of A in position $(i, j)$. The matrix derivative of a differentiable function $f : \R^{m\times n} \to \R$ is denoted $\frac{\partial}{\partial X}f(X)$, where $\left[ \frac{\partial}{\partial X}f(X)\right]_{ij} = \partial f(X) / \partial X_{i,j}$. The projection of a variable $y \in \mathcal Y$ onto a set $\X \subseteq \mathcal Y$ is denoted by $\underset{\X}{\Pi}(y)$. 
$\L_2$ means the space of square-integrable random variables with the associated inner product $\langle \cdot, \cdot \rangle_{\L_2}$, and semi-norm $\|\cdot\|_{\L_2}$.
The set of Gaussian variables with mean $m$ and covariance $\Sigma$ is denoted $\N(m, \Sigma)$ and $\E[\cdot]$ denotes the expectation operator.
\subsection{Online Convex Optimization}
The online convex optimization setting is a repeated leader-follower game between a minimizing player and an adversary. At each time-step $t$, the minimizing player first decides $x_t$ from some compact convex set $\X$. The adversary then observes $x_t$ and selects a convex loss function $f_t$ that is uniformly bounded and has bounded gradients. The minimizing player pays $f_t(x_t)$ and learns the entire function $f_t$. The goal is to minimize the sum, $\sum_{t=1}^T f_t(x_t)$ over an arbitrary sequence of differentiable convex loss functions $f_1, f_2, \ldots, f_T$ with bounded derivatives.
\onote{
Recently, online convex optimization has seen an increasing number of applications across different fields including generator scheduling in smart grids~\cite{Narayanaswamy2012Online}, thermal management of multiprocessors~\cite{Zanini2010OCO}, demand steering via real-time electricity pricing~\cite{Kim2014Pricing} and on-policy learning of optimal control policies with linear dynamics~\cite{Li2021Online, Chen2021Blackbox, Hazan2020Nonstochastic, Cohen2018OnlineLQ}.
}
The performance measure is regret against the optimal policy in hindsight,
\[
    R(T) = \sum_{t=1}^T f_t(x_t) - \min_{x \in \X}\sum_{t=1}^T f_t(x).
\]

Online gradient descent, introduced by Zinkevich~\cite{Zinkevich2003Online}, is a simple, general yet efficient algorithm that applies to many online convex optimization problems and is given in Algorithm~\ref{alg:OGD}. Online gradient descent attains the asymptotic lower bounds $\Omega(DG\sqrt{T})$ and $O(\log T)$ for convex functions and $\alpha$-strongly convex functions, respectively. $D$ bounds the diameter of the feasible set, and $G$ bounds the norm of the gradient.

\begin{algorithm}[tb]
   \caption{Online Gradient Descent}
   \label{alg:OGD}
\begin{algorithmic}
    \STATE {\bfseries Input:} Convex set $\X$, $T$, $x_1 \in \X$ step-sizes $\{\eta_t\}$
    \FOR{$t=1$ {\bfseries to} $T$}
       \STATE Play $x_t$, observe $f_t$ and pay $f_t(x_t))$
       \STATE Update and project $x_{t+1} = \underset{\X}{\Pi}(x_t - \eta_t\nabla f_t(x_t))$
   \ENDFOR
\end{algorithmic}
\end{algorithm}

We will work with matrix-valued variables and strongly convex functions for the team decision problem, using the below definition of strong convexity.
\begin{definition}[Strong Convexity, matrix case]
    We say that the differentiable function $f : \X \subseteq \R^{m \times p} \to \R$ is \emph{strongly convex} with \emph{coefficient} $\alpha$ if for all $X,Y \in \X$,
    \[
    f(X) - f(Y) \leq \left \langle \frac{\partial}{\partial X} f(X), (X-Y) \right \rangle_F - \frac{\alpha}{2}\|x - y\|_F^2.
    \]
\end{definition}
An equivalent characterization is to require that the function $X \mapsto f(X) - \frac{\alpha}{2}\|X\|_F^2$ is convex, \cite{Bubeck2015Convex}. We refer the reader to \cite{Hazan2019Online} for more details on online convex optimization.

\subsection{Bandit and Zeroth-Order Optimization}
The minimizing player observes only the incurred cost $f_t(x_t)$ after each round in the bandit setting, rather than the gradient. This necessitates \emph{exploration} to learn properties of the loss functions, such as gradients, to accelerate optimization. Derivative-free methods have a long history in stochastic optimization. Tight convergence rates for strongly convex functions were obtained in~\cite{Rakhlin2011Gradient} in the first- and~\cite{Shamir2013complexity} in the zeroth-order setting. Bandit feedback was introduced to the online convex optimization setting in \cite{Flaxman2004online} where the authors used a one-point gradient estimate. Their method has asymptotic regret upper bounded by $O(T^{3/4})$. 

\subsection{Stochastic Team Decision theory}
The stochastic team-decision problem, is to solve
\begin{equation}
	\begin{aligned}
		\underset{\mu}{\text{minimize}} \quad &\E [\|z\|_2^2]\\
		\text{subject to:}\quad & z = Hx + Du\\
		& y_i = C_ix + v_i \\
		& u_i = \mu_i(y_i), \quad i = 1, \ldots, N.
	\end{aligned}
	\label{eq:stoch_prob}
\end{equation}

In \eqref{eq:stoch_prob}, $x \sim \N(0, V_{xx})$ and $v \sim \N(0, V_{vv})$ are independent Gaussian variables taking values in $\R^{n}$ and $\R^{p}$ respectively. 
$u_i \in \R^{m_i}$ denotes a player, and the players $u_1, \ldots, u_N$ make up a team.
The function $\mu(\cdot) : \R^p \to \R^m$ represents the decision function of the team, that is, $
    \mu(Cx) = \bmat{
    \mu_1(y_1)^\top &
    \cdots &
    \mu_1(y_N)^\top
    }^\top.
$
We further assume that $D^\top D \in \S^m_{++}$ where $m = m_1 + \cdots + m_N$. 
Radner~\cite{radner} showed that the optimal decision functions $\mu^\star_i$ are unique and linear in $y_i$. This motivates the search over linear policies in our problem set-up.

%% file: sections/problem.tex
We aim to learn the optimal decision policy through repeated interactions with the environment. 
At each time-step $t$, each team-member will decide on a decision policy $ K_i^t$, receive a noisy partial observation of the system state, $y_i^t$, play the decision $u_i^t = K_i^t y_i^t$.
The team incurs the loss $l_t(K_t) = \|z_t\|_2^2$, generated by
\begin{equation}\label{eq:z}
\begin{aligned}
	z_t & = Hx_t + Du_t, & y^t_i & = C^t_ix + v^t_i \\
    u^t_i & = K_i^t y^t_i, & i = 1, \ldots, N.
\end{aligned}
\end{equation}
The objective is to minimize the sum of the losses, $J = \sum_{t = 1}^T l_t(K_t)$, while maintaining $K_t \in \K$, learning good policies locally.
$\K$ is the set of real-valued block-diagonal matrices of appropriate dimensions,
\begin{multline}
    \K := \{K :  K = \Diag(K_1, \ldots, K_N), K_i \in \R^{m_i \times p_i} \}.
    \label{eq:K}
\end{multline}

We summarize the interaction in Algorithm~\ref{alg:learn}.
Going forward we make the following assumptions.
\begin{assumption}
\label{ass:bounded_moments}
$x_t$ and $v_t$ have finite covariance matrices $\E [x_tx_t^\top] = V_{xx}$ and $\E [v_tv_t^\top] = V_{vv}$ and bounded fourth order moments so that $\E [(x_t^\top x_t)^2] \leq \kappa_x$ and $\E [(v_t^\top v_t)^2] \leq \kappa_v$.
\end{assumption}
\begin{assumption}
\label{ass:strongly_convex}
\[
    \sigma_\text{min}(D^\top D) (\sigma_\text{min}(CV_{xx}C^\top) + \sigma_\text{min}(V_{vv})) > 0.
\]
\end{assumption}

Assumption~\ref{ass:bounded_moments} is motivated by the fact that the variance of an estimator of the derivative $\frac{\partial}{\partial K}J(K)$ will contain fourth-order moments. Assumption \ref{ass:strongly_convex} is to the losses being strongly convex in expectation, which is summarized in Proposition~\ref{prop:strongly_convex}. 
Finally, we restrict our search to policies with an apriori supplied bound.
\begin{assumption}
    \label{ass:bk}
    A bound $b_K$ on $\|K\|_2$ is supplied by an oracle.
\end{assumption}
Let $K^\star$ be the best policy in hindsight,
\begin{equation}
    K^\star = \argmin_{K\in\K, \|K\|_2 \leq b_K}\sum_{t=1}^T l_t(K).
    \label{eq:Kstar}
\end{equation}
We measure performance as expected regret, 
\begin{equation}
    \E[{R}(T)] = \E\left[\sum_{t=1}^T l_t(K_t) - \sum_{t=1}^T l_t(K^\star) \right].
    \label{eq:regret}
\end{equation}
\begin{algorithm}[tb]
   \caption{ Learning with repeated interactions.}
   \label{alg:learn}
\begin{algorithmic}
    \FOR{$t=1$ {\bfseries to} $T$}
       \STATE Sample $x_t \sim \N(0, V_{xx})$ and $v_t \sim \N(0, V_{vv})$
        \STATE Agents $1, 2, ..., N$ observes $y^t_1, \ldots, y^t_N$ as in \eqref{eq:z}, respectively.
        \STATE The agents play $K^t_1, \ldots, K^t_N$, respectively, and incur a loss $ l_t(K_t) := \|z_t\|_2^2$, with $z_t$ as in \eqref{eq:z}.
        \STATE Each agent $i$ observes either
            \begin{itemize}
        \item the partial derivative, $\frac{\partial}{\partial K_i} (z_t)^\top z_t$, in the gradient-feedback setting,
        \item or the incurred loss, $\|z_t\|_2^2$, in the bandit-feedback setting
    \end{itemize}
    \STATE The agents update their policies $K_i^{t+1}$.
    \ENDFOR
\end{algorithmic}
\end{algorithm}

%% file: sections/stoch.tex
The losses $l_t$ are differentiable with respect to $K_t$ everywhere. In particular, the derivative with respect to agent $i$ can be viewed as a product of the information available to the agent $y_i$, and their contribution to the overall cost, $[D^\top]_i z$.

\begin{proposition}
    \label{prop:dJ}
    $l_t$ is differentiable with respect to $K_i$ and the derivative is
    \[
        \frac{\partial}{\partial K_i} \E[l_t(K)] = \E[\left[2[D^\top]_i z_t(y^t_i)^\top \right].
    \]
\end{proposition}
\begin{proof}
By dominated convergence, we can exchange expectation and differentiation.\footnote{We drop the time-index for readability}
\begin{align*}
    \frac{\partial}{\partial K} \E [z^\top z]  &= \E\left[ \frac{\partial}{\partial K}(Hx + DKy)^\top (Hx + DKy)\right] \\
    & = \E\left[2D^\top zy^\top \right].
\end{align*}


Identifying the local components $\partial / \partial K_i$ completes the proof.
 \end{proof}

%
The phenomenon that certain large changes to the optimization variable can have (almost) negligible effects on the value can make optimization difficult. The right way to quantify this effect on convergence is through \emph{strong convexity}, a property we can exploit to get better regret bounds in online convex optimization~\cite{Hazan2019Online}. In our regret terms, a lower bound on the strong convexity parameter will show up directly as a divisor. The following proposition shows that $\E [l_t]$ is strongly convex as a function of $K$. 

\begin{proposition}
\label{prop:strongly_convex}
$\E [l_t]$ is $\alpha$-strongly convex with constant
\[
    \alpha = 2\sigma_\text{min}(D^\top D) (\sigma_\text{min}(CV_{xx}C^\top) + \sigma_\text{min}(V_{vv})).
\]
\end{proposition}
\begin{proof}
    We will verify that $\E [l_t](K) - \frac{\alpha}{2}\|K\|^2_F$ is convex. $\E [l_t]$ is a quadratic function of $K$ and
    \[
    \begin{aligned}
        \E [l_t](K) 
            & = \|Hx\|_{\L_2}^2 + 2\langle Hx, DKy \rangle_{\L_2} + \|DKy\|^2_{\L_2}.
        \end{aligned}
    \]
    
    Which is convex if and only if $\|DKy\|^2_{\L_2} \geq \frac{\alpha}{2}\|K\|^2_F$. Consider,
    \begin{align*}
        \|DKy\|^2_{\L_2} & \geq \sigma_\text{min}(D^\top D)\|KCx + Kv\|^2_{\L_2}
                        \geq \frac{\alpha}{2} \|K\|_F^2
    \end{align*}
\end{proof}
To apply online optimization algorithms to learn the optimal policy through repeated play, we must bound the second and fourth moments of $z$ as we must bound the variance of our derivative estimates. 
We get the following bounds on the second and fourth order moments of $z$ by Assumptions~\ref{ass:bounded_moments} and \ref{ass:bk}.  

\begin{proposition}
For $\|K\|_2 \leq b_K$, the loss $l_t(K)$ in Algorithm~\ref{alg:learn} is bounded from above in expectation, $\E[l_t](K) \leq b_l$, where
\[
    b_l = (\|H\|_2 + \|D\|_2\|C\|_2b_K)^2 \Tr V_{xx} + \|D\|_2^2 b_K^2 \Tr V_{vv}.
\]

Furthermore, $\E[(z_t^\top z_t)^2 \leq \kappa_z]$ where 
\begin{multline}
    \kappa_z = \left(\|H\|_2 +\|D\|_2\|C\|_2b_K + \|D\|_2 b_K \right)^4 \\
    \times \left(\kappa_x + \Tr V_{xx} \Tr V_{vv} + \kappa_v \right)
\end{multline}
\label{prop:stoch_bound_J}
\end{proposition}
\begin{proof}[Proof of Proposition~\ref{prop:stoch_bound_J}]
We start with bounding the value function. Let $\|\cdot \|_{\L_2}$ be the $\L_2$ norm, then 
\begin{align*}
\E[l](K)  &= \|(H + DKC)x + DKv\|_{L_2}^2 \\
&= \|(H + DKC)x\|_{L_2}^2 + \|DKv\|_{L_2}^2,
\end{align*}
as $x$ and $v$ are independent. By the triangle inequality
\[
    \|(H + DKC)x\|_{L_2}^2 \leq (\|H\|_2 + \|D\|_2\|K\|_2\|C\|_2)^2\|x\|^2_{\L_2}.
\]
Treating the term $\|DKv\|_{\L_2}$ similarly and substituting $\|K\|_2 \leq b_K$ and $\|x\|^2_{\L_2} = \Tr V_{xx}$ completes the proof. To prove the second claim, consider 
\begin{align*}
    \E [(z^\top z)^2] & = \E \left[\left \|\bmat{H + DKC & DK} \bmat{x \\v}\right \|_2^4\right] \\
                    & \leq \left \|\bmat{H + DKC & DK} \right \|_2^4 \E [(x^\top x + v^\top v)^2] \\
                    & \leq \left(\|H\|_2 +\|D\|_2\|C\|_2b_K + \|D\|_2 b_K \right)^4 \\
    & \quad \times \left(\kappa_x + \Tr V_{xx} \Tr V_{vv} + \kappa_v \right)
\end{align*}
\end{proof}

%% file: sections/learning.tex
This section describes how to learn the optimal team decision policies using online gradient descent.
Due to the stochastic nature of our problem, we cannot hope to bound the objective function or the gradient for an arbitrary realization. We will modify the analysis to give results when these properties hold in expectation. This means our guarantees hold in expectation and are well suited to analyze stochastic problems. We summarize the upper bound for expected regret for strongly convex functions in Theorem~\ref{thm:GD} The bound is what one would expect; the standard result \cite[Theorem 3.3]{Hazan2019Online} for strongly convex functions holds in expectation against an adaptive adversary.
\begin{theorem}
\label{thm:GD}
    Let $l_1, \ldots, l_T$ be independent random functions $l_t : R^{m \times n} \to \R $ such that $\E [l_t]$ is $\alpha$-strongly convex for all $t = 1, \ldots, T$. Let $\tilde \nabla_t$ be a derivative oracle that is \textbf{consistent} $\E[ \tilde \nabla_t] = \frac{\partial}{\partial K}\E [l_t(K)]$ and has \textbf{bounded variance} $\E \left[ \|\tilde \nabla_t\|_F^2\right] \leq (b_t)^2$ for all $K \in \mathcal{K}$, where $\K$ is convex and compact. Set the step size $\eta_t = \frac{1}{\alpha t}$. Let $K^\star = \argmin_{K\in \K}\sum_{t=1}^T l_t(K)$. Online Gradient Descent, Algorithm~\ref{alg:OGD}, has expected regret
    \begin{equation}
        \E\left[\sum_{t=1}^T\left(l_t(K_t) - l_t(K^\star) \right) \right] \leq\frac{1}{2}\sum_{t=1}^T \frac{b_t^2}{\alpha t}.
    \end{equation}
\end{theorem}

The proof follows the outline in~\cite{Hazan2019Online}, but involves some extra bookkeeping:
\begin{proof}
Let $\F_t = \sigma(J_1, \ldots, J_{t-1})$. Then $K_t$ is a stochastic sequence adapted to $\F_t$. Define for simplicity $\nabla_t = \frac{\partial}{\partial K} \E [l_t(K_t)]$. By strong convexity
\[
2\E\left[l_t(K_t) - l_t(K^\star) | \F_t\right] \leq 2 \langle \nabla_t, K_t - K^\star \rangle_F - \alpha \|K^\star - K_t\|_F^2.
\]
To bound $\langle \nabla_t, K_t - K^\star \rangle_F$, consider
\begin{align*}
    &\E\left[\|K_{t+1} - K^\star\|_F^2 |\F_t\right]  = \E\left[\|\underset{\K}{\Pi}(K_t - \eta_t\tilde\nabla_t) - K^\star\|_F^2 |\F_t\right] \\
    &    \quad \leq \|K_t - K^\star\|_F^2 + \eta_t^2 b_t^2 - 2 \eta_t\langle\nabla_t, K_t - K^\star \rangle_F.
\end{align*}
Taking $\eta_t = \frac{1}{\alpha t}$ and defining $\frac{1}{\eta_0} = 0$, we get $
    2\E\left[\sum_{t=1}^T\left(l_t(K_t) - l_t(K^\star) \right) \right] \leq \sum_{t=1}^T\frac{b_t^2}{\alpha t} $.
%
\end{proof}

We are now ready to apply online gradient descent to learn distributed team decisions.
\subsection{Learning Team Decisions with Partial Gradient Information}
We assume that the designer is aware of a lower bound on the strong convexity parameter, $\lambda$, and upper bound on the operator norm of the optimal policy $b_K$. The resulting algorithm, Algorithm~\ref{alg:fg_feedback}, is a direct extension of online gradient descent. Its behavior is summarized in Theorem~\ref{thm:fg_feedback}.
\begin{algorithm}[tb]
   \caption{Learning with partial gradient information}
   \label{alg:fg_feedback}
\begin{algorithmic}
   \STATE {\bfseries Input:} initial guess $K_0$, bound $b_{K}$, step-sizes $\{\eta_t\}$
   \STATE Each agent plays $u^t_i = K^t_i y^t_i$
   \STATE The team incurs cost $l_t(K_t) = z_t^\top z_t$
   \FOR{$t=0$ {\bfseries to} $T-1$}
       \FOR{$i = 1$ {\bfseries to} $N$}
       \STATE Observe the partial gradient $G^t_i = 2 D_i^\top z_t(y_i^t)^\top$
       \STATE Update $L^{t+1}_i  = K_i^t - \eta_t G^t_i$
       \IF{$\|L_i^{t+1}\|_2 > b_{K}$}
        \STATE  $K_i^{t+1} = L_i^{t+1}/b_{K}$
        \ELSE 
        \STATE$K_i^{t+1} = L_i^{t+1}$
        \ENDIF
       \ENDFOR
   \ENDFOR
\end{algorithmic}
\end{algorithm}

\begin{theorem}[Partial-Gradient Feedback]
\label{thm:fg_feedback}
 Assume that Assumptions \ref{ass:bounded_moments}, \ref{ass:strongly_convex} and \ref{ass:bk} hold. Then, Algorithm~\ref{alg:fg_feedback} with step-size $\eta_t = \frac{1}{\lambda t}$ for any $0 < \lambda \leq \alpha$, where $\alpha$ is the strong-convexity parameter in Proposition~\ref{prop:strongly_convex}, has bounded expected regret against the optimal policy $K^\star$ defined in \eqref{eq:Kstar}. The bound is given by
\begin{equation}
    \E[R(T)] = \sum_{t=1}^T \E[l_t(K_t) - l_t(K^\star)] \leq \frac{b_G^2}{2\lambda}(1 + \log(T)).
    \label{eq:fg_regret}
\end{equation}

The constant $b_G$ in \eqref{eq:fg_regret} is given by
\begin{multline*}
    b^2_G = 4\|D\|_2^2(\|H\|_2 + b_K\|D\|_2(\|C\|_2  + 1))^2(\|C\|_2 + 1)^2\\
    (\kappa_x + 2 \Tr V_{xx} \Tr V_{vv} + \kappa_v)
\end{multline*}
\end{theorem}

The regret bound is equivalent to that of online gradient descent in the convex optimization setting, where $b_G$ takes the place of the bound on the gradient. The difference is that the bound holds in expectation and that $b_G^2$ is a bound on the second moment of the gradient estimator. Before proving Theorem~\ref{thm:fg_feedback} we need the following lemma to characterize the gradient estimate.
\begin{lemma}
\label{lemma:stoch_bound_dJ}
For $\|K\|_2 \leq b_K$, the gradient estimate $G^t_i := 2 [D^\top]_i z_t (y_i^t)^\top$ is \textbf{consistent:} $\E [G^t_i] = \frac{\partial}{\partial K_i}\E[l_t(K)]$, and has \textbf{bounded variance}: $\E \left[\|G_t\|^2_F\right] \leq b^2_G$, where $G_t = \Diag(G_1^t, \ldots, G_N^t)$ and $b_G$ satisfies
\begin{multline}
    b^2_G = 4\|D\|_2^2(\|H\|_2 +  b_K\|D\|_2(\|C\|_2 + 1))^2(\|C\|_2 + 1)^2\\
    (\kappa_x + 2 \Tr V_{xx} \Tr V_{vv} + \kappa_v).
\end{multline}
\end{lemma}
\begin{proof}[Proof of Theorem~\ref{thm:fg_feedback}.]
Since all agents have the same loss functions, the partial gradient update is equivalent to a full gradient update. The result thus follows directly from Theorem~\ref{thm:GD} with the covariance-bounded gradient oracle in Lemma~\ref{lemma:stoch_bound_dJ} and the strong convexity coefficient from Proposition~\ref{prop:strongly_convex}. 
\end{proof}

\subsection{Learning Team Decisions with Bandit Feedback}
Towards constructing an estimator for the derivative, in addition to requiring the estimate to be consistent and have bounded variance, we insist that each agent must be able to compute her estimate independently. The last requirement invalidates the one-point estimate used in \cite{Flaxman2004online} as sampling from the unit sphere would require communication between agents. In \cite{Shamir2013complexity}, the authors found that sampling uniformly and independently from the unit hypercube leads to consistent and bounded estimators for quadratic problems. Sampling from the hypercube reduces to sampling independent Rademacher variables coordinate-wise and can be done in a distributed fashion. Algorithm~\ref{alg:bandit_feedback} is constructed by applying a matrix version of the estimate from \cite{Shamir2013complexity} and shrinking the exploration parameter $\epsilon_t$ each time-step. The regret properties of Algorithm~\ref{alg:bandit_feedback} is summarized in Theorem~\ref{thm:bandit_feedback}.

\begin{algorithm}[tb]
   \caption{Learning with bandit feedback}
   \label{alg:bandit_feedback}
\begin{algorithmic}
   \STATE {\bfseries Input:} initial guess $K_0$, bound $b_{K}$, step-sizes $\{\eta_t\}$ and exploration parameters $\{\epsilon_t\}$
   \FOR{$t=1$ {\bfseries to} $T$}
       \FOR{$i = 1$ {\bfseries to} $N$}
       \STATE Sample $R_i^t \in \{-1, 1\}^{m_i \times p_i} \sim \text{Uniform}$
       \STATE Play $u^t_i = (K^t_i + R^t_i\mathcal E^t_i) y^t_i$ where $\mathcal E^t_i = \frac{\epsilon_t}{\sqrt{m_i p_i}}$
    \ENDFOR
    \STATE The team incurs loss $l_t(K_t + R^t\mathcal E^t) = z_t^\top z_t$
    \FOR{$i = 1$ {\bfseries to} $N$} 
       \STATE Observe loss $l_t = z_t^\top z_t$ 
       \STATE Set gradient estimate $\tilde G^t_i = l_t R^t_i (\mathcal{E}^t_i)^{-1}$
       \STATE Update $L^{t+1}_i  = K^t_i - \eta_t \tilde G^t_i$
       \IF{$\|L_i^{t+1}\|_2 > b_{K}$}
        \STATE  $K_i^{t+1} = L_i^{t+1}/b_{K}$
        \ELSE 
        \STATE$K_i^{t+1} = L_i^{t+1}$
        \ENDIF
       \ENDFOR
   \ENDFOR
\end{algorithmic}
\end{algorithm} 

\begin{theorem}[Bandit Feedback]
\label{thm:bandit_feedback}
Assume that Assumptions \ref{ass:bounded_moments}, \ref{ass:strongly_convex} and \ref{ass:bk} hold. Then, Algorithm~\ref{alg:bandit_feedback} with step-sizes $\eta_t = \frac{1}{\lambda t}$ for any $0 < \lambda \leq \alpha$ where $\alpha$ is the strong-convexity parameter in Proposition~\ref{prop:strongly_convex}, and exploration parameters $\epsilon_t = t^{-1/4} \left( \sum_{i=1}^N m_i^2p_i^2 \right)^{-1/4} $ has bounded expected regret against the optimal policy $K^\star$ defined in \eqref{eq:Kstar}. The bound is given by
\begin{multline}
    \sum_{t=1}^T\E[l_t(\tilde K_t) - l_t(K^\star)] \leq \\
    2\left( M_1 + \frac{M_2}{\lambda}\right)\left(\sum_{i=1}^N m_i^2p_i^2\right)^{1/2}\sqrt{T}.
    \label{eq:bandit_regret}
\end{multline}

In \eqref{eq:bandit_regret} $\tilde K_t = K_t + R_t \mathcal E_t$ is the policy played by the agents at time-step $t$. The problem-dependent constants $M_1$ and $M_2$ are given by
\begin{align*}
M_1 & = \|D\|_2^2\left(\|C\|_2^2\Tr V_{xx} + \Tr V_{vv}\right) \\
M_2 & = \left(\|H\|_2 + \|D\|_2(b_K + 1)(\|C\|_2 + 1)\right)^4 \\
    & \quad  \times (\kappa_x + 2 \Tr V_{xx} \Tr V_{vv} + \kappa_v).
\end{align*}
\end{theorem}

To prove \ref{thm:bandit_feedback} we need the following lemma, which states that the gradient estimator used in Algorithm~\ref{alg:bandit_feedback} is consistent and has bounded variance.
\begin{lemma}[Variant of lemma 2 in \cite{Shamir2013complexity}]
    \label{lemma:stoch_bound_Gtilde}
	Let $R_i \in \{-1, 1\}^{m_i \times p_i}$ be independent random variables following uniform distributions. Let $R = \Diag\{R_1, R_2, \ldots, R_N\}$, $\mathcal E = \epsilon\Diag\{I/\sqrt{m_1p_1}, I/\sqrt{m_2p_2}, \ldots, I/\sqrt{m_Np_N}\}$. Define the zeroth-order gradient estimator
	\[
		\tilde G^t_i := l_t\left(K + R\mathcal E\right) R_i\mathcal E^{-1}_i,
	\]
	and let $\tilde G_t := \Diag\{\tilde G^t_1,\ldots, \tilde G^t_N\}$. Under assumptions~\ref{ass:bounded_moments}--\ref{ass:bk},  $\tilde G_t$ is \textbf{consistent:} $\E_{R, x, v} \left[\tilde G^t_i\right] = \frac{\partial}{\partial K_i} J(K)$ and has \textbf{bounded variance} $\E_{R, x, v}\left[ \|\tilde G_t\|_F^2\right] \leq \tilde b^2_G$. The bound, $\tilde b_G$ can be taken as
		\begin{multline}
            \tilde b^2_G = \left(\|H\|_2 + \|D\|_2(b_K + \epsilon))(\|C\|_2 + 1)\right)^4 \\
            \times (\kappa_x + 2 \Tr V_{xx} \Tr V_{vv} + \kappa_v)\sum_{i=1}^N m_i^2p_i^2 / \epsilon^2.
		\end{multline}
\end{lemma}
Note that the bound $\tilde b_G$ is decreasing in the exploration parameter $\epsilon_t$, leading to an exploration/exploitation trade-off. The choice of $\epsilon_t$ minimizes the regret asymptotic upper bound. 
\begin{proof}[Proof of Theorem~\ref{thm:bandit_feedback}.]
We will first quantify the added loss due to the perturbation term $R_t\mathcal E_t$. Let $\tilde K_t = K_t + R_t\mathcal E_t$, then 
\begin{align*}
    \E [l_t(\tilde K_t)] & = \E\left[ \|Hx_t + D(K_t + R_t\mathcal E_t)Cy_t \|_2^2\right] \\
                    & = \E\Big[\|Hx_t + DKy_t\|_2^2 + \|D R_t\mathcal E_t y_t\|_2^2\\ 
                    & \quad + 2(Hx_t + DKy_t)^\top R_t\mathcal E_t y_t \Big].
\end{align*}
By the first property of Lemma~\ref{lemma:R} in the Appendix, we know that $\E [R_t] = 0$. Applying the fifth property we conclude that
\begin{multline*}
    \E [l_t(\tilde K_t) ]\leq \E[l_t(K_t)]  + \epsilon_t^2\|D\|_2^2\left(\|C\|_2^2\Tr V_{xx} + \Tr V_{vv}\right).
\end{multline*}
Combining this with Lemma~\ref{thm:GD}, we get
\begin{multline*}
    \sum_{t=1}^T \E[J(\tilde K_t) - J(K^\star)]  \leq \sum_{t=1}^T  \frac{(\tilde b_G^t)^2}{\lambda t} + M_1\sum_{t=1}^T \epsilon_t^2.
\end{multline*}
Substituting $\epsilon_t$ into $\tilde b_G^t$ from Lemma~\ref{lemma:stoch_bound_Gtilde} and the inequality $\sum_{t=1}^T \frac{1}{\sqrt{t}} \leq 2 \sqrt{T}$ completes the proof.
\end{proof}

%% file: sections/example.tex
In Fig.~\ref{fig:example} we apply the algorithms to \cite[Example 4.1]{Gattami2007Optimal} for two players, where $C_1 = C_2 = 1$, $x, v_1, v_2 \sim \mathcal N(0, 1)$ and
\[
H = \begin{bmatrix} 1 & 0 & 0 \end{bmatrix}^\top, \quad D = \bmat{1 & 1 & 0 \\ 1 & 0 & 1}^\top.
\]
Regret is bounded for Algorithm~\ref{alg:fg_feedback} by $46000(1 + \log(t))$ and for Algorithm~\ref{alg:bandit_feedback} by $1.42\cdot 10^6 \sqrt{T}$. The results from the 1280 simulations\footnote{For a Julia implementation, See~\url{https://github.com/kjellqvist/LearningTeamDecisions.jl}} in Fig.~\ref{fig:example} indicates far better performance.
\begin{figure}
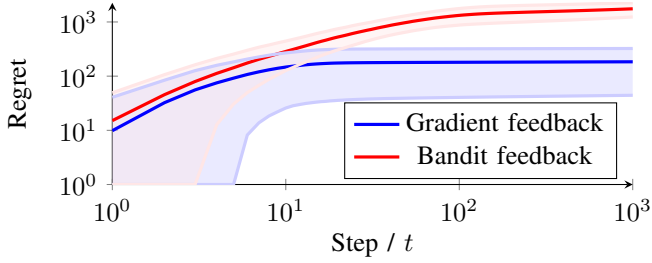

    \vspace{1em}
    \centering
    \include{figures/example}
    \caption{The average (solid lines) $\pm$ one standard deviation (shaded area) from 1280 simulations of Example 4.1 in~\cite{Gattami2007Optimal} using Algorithm~\ref{alg:fg_feedback} (gradient feedback, blue) and Algorithm~\ref{alg:bandit_feedback} (bandit feedback, red).}
    \label{fig:example}
\end{figure}

%% file: figures/example.tex
    \pgfplotstableread[col sep=comma]{figures/stats.csv}{\datatable}
    \begin{tikzpicture}
       \begin{axis}[%
            xlabel={Step / $t$},
            ylabel={Regret},
            width={8.5cm},
            height={4cm},
            ymode={log},
            xmode={log},
            ymin = 1,
            xmin = 1,
            axis x line={bottom},
            axis y line={left},
            xminorticks = {false},
            yminorticks = {false},
            legend pos = {south east},
        ]
        \addplot[color=blue, very thick] table[x=t, y=avg_gradient] from \datatable;
        \addlegendentry{Gradient feedback}
        \addplot[color=red, very thick] table[x=t, y=avg_bandit] from \datatable;
        \addlegendentry{Bandit feedback}
        
        \addplot[color=blue!20, very thick, name path= upper_grad] table[x=t, y expr={\thisrow{avg_gradient} + \thisrow{std_gradient}}] from \datatable;
        \addplot[color=blue!20, very thick, name path= lower_grad] table[x=t, y expr={max(1, \thisrow{avg_gradient} - \thisrow{std_gradient})}] from \datatable;
        \addplot[blue!20, opacity=0.4] fill between[of=lower_grad and upper_grad];
        
        \addplot[color=red!10, very thick, name path= upper_bandit] table[x=t, y expr={\thisrow{avg_bandit} + \thisrow{std_bandit}}] from \datatable;
        \addplot[color=red!10, very thick, name path= lower_bandit] table[x=t, y expr={max(1, \thisrow{avg_bandit} - \thisrow{std_bandit})}] from \datatable;
        
        \addplot[red!10, opacity=0.4] fill between[of=lower_bandit and upper_bandit];
       \end{axis}
    \end{tikzpicture}

%% file: sections/conclusions.tex
We have proposed algorithms that efficiently learn optimal team decisions in a decentralized manner without knowing the problem parameters. The exploration required with bandit feedback gives worse asymptotic regret, both with respect to time and the number of parameters to be learned. Our work gives a first approach, and there are several interesting open questions to answer. Interesting directions for future research include learning when the covariance matrices change over time, applications to feedback control of dynamical systems, and empirical convergence studies.

%% file: sections/appendix.tex
\label{app:rademacher}
\begin{lemma}
\label{lemma:R}
Let $R_i \in \{-1, 1\}^{m_i \times p_i}$ for $i = 1, \ldots, N$ be independent random variables following uniform distributions, and take $R = \Diag\{R_1, R_2, \ldots, R_N\}$. Define $m = m_1 + \cdots + m_N$ and $p = p_1 + \cdot p_N$. Define the set
\begin{multline*}
    \mathcal I_R := \Big\{ (i, k, l) \in \mathbb N^3 : i \in \{1, \ldots, N \}, k \in \{1, \ldots, m_i \}, \\
    l \in \{1, \ldots, p_i \} \Big \}.
\end{multline*}
Let $(i, k, l),\ (i', k', l')$ and $(\hat i, \hat k, \hat l) \in \mathcal I_R$. Them the following  hold
\begin{enumerate}
		\item $\E [R_i(k,l)] = 0$,
		\item $\E [R_i(k,l)R_{i'}(k', l')] = \delta_{(i,k,l) = (i',k',l')}$,
		\item $\E [R_i(k,l)R_{i'}(k',l')R_{\hat i}(\hat k, \hat l)] = 0$,
		\item $\E\left[\Tr (AR^\top)R_i\right] = [A]_i$ for all $A \in \R^{m \times p}$,
		\item $\|R_i\|_F = \sqrt{m_ip_i}$.
\end{enumerate}

\end{lemma}